\documentclass[12pt]{amsart}

\usepackage{amscd}%
\usepackage{amsmath}%
\usepackage{amsfonts}%
\usepackage{amssymb}%
\usepackage{graphicx}
\usepackage[T2A]{fontenc}
\usepackage[russian,english]{babel}
\usepackage[all]{xy}

\theoremstyle{plain}
\newtheorem{theorem}{Theorem}[section]
\newtheorem{lemma}{Lemma}[section]

\newtheorem{corollary}{Corollary}[section]

\theoremstyle{definition}

\theoremstyle{definition}
\newtoks{\thehExample}
\newtheorem*{Example}{\the\thehExample}
\newenvironment{Ex}[1]{\thehExample{#1}\begin{Example}}{\end{Example}}

\begin{document}
\selectlanguage{english}
\title[Nets of graded $C^*$-algebras]
{Nets of graded $C^*$-algebras over partially ordered sets}

\author{S.~A.~Grigoryan, E.~V.~Lipacheva, A.~S.~Sitdikov}

\address{Kazan State Power \\
Engineering University,\\
Krasnoselskaya st., 51 \\420066,  Kazan, Russia}

\email{gsuren@inbox.ru}

\email{elipacheva@gmail.com}

\email{airat$_{-}$vm@rambler.ru}

\maketitle

\begin{abstract}
The paper deals with $C^*$-algebras generated by a net of Hilbert
spaces over a partially ordered set.
 The family of those algebras constitutes a net of $C^*$-algebras over the same  set.
It is shown that every such an algebra is graded by the first
homotopy group of the partially ordered set. We consider inductive
systems of $C^*$-algebras  and their   limits over maximal directed
subsets. We also study properties of morphisms for nets of Hilbert
spaces as well as nets of $C^*$-algebras.

\end{abstract}

\

\emph{keywords:}{$C^*$-algebra, graded $C^*$-algebra, partially
ordered set, net of $C^*$-algebras,  net of Hilbert spaces, path
semigroup, the first homotopy group, inductive limit }

\maketitle

\section{Introduction}

The paper is devoted to the construction and the study of nets consisting of $C^*$-algebras
generated by nets of Hilbert spaces over  partially ordered sets. One of the directions in  the applications of such nets is algebraic quantum field theory. In a spacetime the family of all open bounded regions is a partially
ordered set  under the inclusion relation
\cite{HK,H,Khoruzhi}. One associates to these regions  the $C^*$-algebras
of local observables which can be measured in the pertinent regions. The family of all those algebras indexed by the regions of a spacetime is called a net of
 $C^*$-algebras.

The main objects of the study in the paper are the local
$C^*$\nobreakdash-algebras (see \S~5) generated by a triple
$$
\big(K,H_a,\gamma_{ba}\big)_{a\leq b\in
K},
$$
where $K$
is a partially ordered set, $H_a$
is a Hilbert space and
$\gamma_{ba}:H_a\rightarrow H_b$ is an isometric embedding. According to papers
\cite{R,V,RuzziVasselli}, we call that triple a net of Hilbert spaces over $K$. The family of  the $C^*$-algebras  constitutes a net over the same set $K$. Each algebra in this net is graded by the first homotopy group $\pi_1(K)$ for the partially ordered set
$K$.

Moreover, we introduce the notion of the corona for a net consisting of the local  $C^*$-algebras. The algebras in the corona are called the quasi-local algebras. It is shown that
these algebras are also  $\pi_1(K)$-graded.

The motivation for our work comes from papers \cite{R,V,RuzziVasselli} in which
 the nets of the $C^*$-algebras of observables are studied for a curved spacetime and a~spacetime manifold with specific topological features.

We have studied earlier the $C^*$-algebras generated by representations of ordered semigroups \cite{GS, AGL, LO1, LO2, LO3, GL}.
The present paper is a continuation of the study begun in the article \cite{GLS}.  There we  dealt with the $C^*$-algebra generated by the path semigroup  in a partially ordered set.

\section{Paths and loops on a partially ordered set}

Let $K$
be  a partially ordered set with an order relation~$\leq$, which is reflexive, antisymmetric and transitive.
Elements $a$ and $b$ are said to be \emph{comparable} in
$K$, if $a\leq b$ or $b\leq a.$ The set $K$ is said to be
\emph{upward directed}, if for any $a,b\in K$ there exists
$c\in K$ such that $a\leq c$ and $b\leq c$.

Further, we consider paths on $K$. Ordered pairs $(b,a)$
for $b\leq a$ and  $\overline{(b,a)}$ for $b\geq a$ are called
\emph{elementary paths} on $K$.   We also define \emph{the reverse elementary paths}
$s^{-1}=\overline{(a,b)}$ for $s=(b,a)$ and $s^{-1}=(a,b)$ for
$s=\overline{(b,a)}$. The elements $\partial_1s=a$ and $\partial_0s=b$ are called,
respectively, \emph{the starting point} and  \emph{the ending point} of the elementary paths. A pair $(a,a)=\overline{(a,a)}=i_a$ is called
\emph{a trivial path}.

Throughout we consider sequences of elementary paths of the following form:
$$
\overline{p}=s_n\ast s_{n-1}\ast\ldots\ast s_1,
$$
where $\partial_0s_{i-1}=\partial_1s_i$ for $i=2,\ldots,n.$
Here, the elements
$\partial_1\overline{p}=\partial_1s_1$ and $\partial_0\overline{p}=\partial_0s_n$
are, respectively,  the starting point and  the ending point of the sequence $\overline{p}$.
The reverse sequence defined as the sequence
$$
\overline{p}^{-1}=s_1^{-1}\ast s_{2}^{-1}\ast\ldots\ast
s_n^{-1}.
$$

Extending the operation ``$\ast$'', we define the multiplication operation for the sequences of
elementary paths
$
\overline{p}=s_n\ast\ldots\ast s_k$, and
$\overline{q}=s_{k-1}\ast\ldots\ast s_1
$
satisfying the condition
$
\partial_1s_k=\partial_1\overline{p}=\partial_0\overline{q}=\partial_0s_{k-1}
$
as follows:
$
\overline{p}\ast \overline{q}=s_n\ast\ldots\ast s_k\ast s_{k-1}\ast\ldots\ast
s_1.
$

We denote by $\overline{S}$ the set of all sequences of elementary paths endowed with the operation~``$\ast$''. It is clear that the operation
``$\ast$'' is associative.

Let us define an equivalence relation on the set $\overline{S}$. To this end, for
all elements $a,b,c\in K$ such that $a\leq b\leq c$, we put
\begin{align}
  (a,b)\ast(b,c)&\sim(a,c);\label{sim1}
    \\
  \overline{(c,b)}\ast\overline{(b,a)}&\sim\overline{(c,a)};\label{sim2}
    \\
(a,b)\ast \overline{(b,a)}&\sim i_a;\label{sim3}
\\
\overline{(b,a)}\ast (a,b)&\sim i_b.\label{sim4}
\end{align}

It is worth noting that  (\ref{sim1})--(\ref{sim4}) imply the following equivalences:
$$
(a,b)\ast i_b\sim(a,b);
$$
$$
  i_a\ast(a,b)\sim(a,b);
$$
$$
\overline{(b,a)}\ast i_a\sim\overline{(b,a)};
$$
$$
 i_b\ast\overline{(b,a)}\sim\overline{(b,a)};
$$
$$
i_a\ast i_a\sim i_a.
$$

We put $\overline{p}\sim\overline{q}$, where
$\overline{p},\overline{q}\in \overline{S}$, if the sequence $\overline{p}$
can be obtained from the sequence $\overline{q}$ (and $\overline{q}$  from
$\overline{p}$) by means of a finite number of relations
(\ref{sim1})--(\ref{sim4}).

One can easily verify that the following properties are fulfilled in $\overline{S}$:

1. for every sequence $\overline{p}\in \overline{S}$
with $\partial_0\overline{p}=a$ and $\partial_1\overline{p}=b$ the relations
 $\overline{p}^{-1}\ast \overline{p}\sim
i_b,\hspace{0.2cm} \overline{p}\ast \overline{p}^{-1}\sim i_a$ hold;

2.  for every sequence $\overline{p}\in \overline{S},$
with $\partial_0\overline{p}=a$ and $\partial_1\overline{p}=b$
the relations $i_a\ast
\overline{p}\sim\overline{p}\sim\overline{p}\ast i_b$ hold.

Let $p=[\overline{p}]$ be an equivalence class containing a sequence
of elementary paths $\overline{p}$. For equivalence classes $p$ and $q$ we define the multiplication operation ``$\ast$''  as follows: if
$\partial_1\overline{p}=\partial_0\overline{q}$ then we set
$$
p\ast q=[\overline{p}\ast\overline{q}]=\{\overline{s}\sim\overline{p}\ast\overline{q} \
| \ \overline{s}\in\overline{S}\}.
$$

Further, let us consider the quotient set of  $\overline{S}$ by the equivalence relation. It is
obvious that the quotient set $\overline{S}/_{\sim}$ is a groupoid. Adding a formal symbol
 0 and putting
$$
p\ast 0=0, \quad 0\ast p=0
$$
 for every
$p\in \overline{S}/_{\sim}$, one may turn the groupoid $\overline{S}/_{\sim}$
 into a semigroup denoted by
 $$
S=\overline{S}/_{\sim}\cup\{0\}.
$$
 Then for every $p,q\in S$ we have
$$
p\ast q=\begin{cases}
       [\overline{p}\ast \overline{q}], & \mbox{if $p\neq 0$, $q\neq 0$ and
       $\partial_1\overline{p}=\partial_0\overline{q}$,}\\
       0, & \mbox{otherwise.}
       \end{cases}
$$

The semigroup $S$ is called \emph{the path semigroup}. An element $p\in S$ is called  \emph{a path} from the point
 $\partial_1p=\partial_1\overline{p}$ to the point
$\partial_0p=\partial_0\overline{p}$. Thus, we identify
all equivalent sequences of elementary paths. That equivalence class is called a path on $K$.

An element $\overline{p}\in \overline{S}$ is called \textit{a loop in
$\overline{S}$} if the equality $\partial _{0} \overline{p}=\partial _{1}
\overline{p}$ holds. For $a\in K$ we denote by $\overline{G}_a$
the set of all loops whose base point is $a$. The set
$\overline{G}_a$ is a semigroup. For $a\in K$ we denote by  $G_{a}$ the set of all equivalence classes of loops whose base point is $a$. In~\cite{GLS} it is shown
$G_a$ is a subgroup in $S$  with the unit $[i_a]$  and other properties of $G_a$ and $S$ are described. In particular, it is proved that if $K$ is an upward directed set then $G_a$ is trivial.

The set $K$ is said to be \emph{path connected} provided that for every  $a,b\in K$
there exists a path $p$ such that $\partial_0p=a,$ $\partial_1p=b.$
 It is shown in \cite{GLS} that for a path-connected set $K$ the isomorphism $G_a\cong G_b$ holds whenever $a,b\in K$. In particular, the mapping
 $\sigma_{ba}:G_a\rightarrow G_b$ defined by the formula
\begin{equation}\label{sigma}
\sigma_{ba}(p)=[\overline{(b,a)}\ast \overline{p}\ast
(a,b)]
\end{equation}
is isomorphism, where $a\leq b$, $\overline{p}\in p$.

In what follows, we assume that the set $K$
is path connected.

The notion of \emph{the first homotopy {\rm(}fundamental\/{\rm)} group} $\pi_1(K)$ for a partially ordered set $K$ is given in \cite{R, RuzziVasselli}. The group~$\pi_1(K)$ is a quotient set of  the set of all paths on $K$ that start and end at the same point by the homotopy equivalence relation. Two paths are  said to be \emph{homotopy equivalent} if one can
be obtained from the other by a finite number of elementary deformations (see \cite{R,
RuzziVasselli}).

\begin{theorem} For every $a\in K$ there exists an isomorphism $\pi_1(K)\cong G_a$.
\end{theorem}
\begin{proof} To prove the theorem it is sufficient to show that the equivalence relation
given by formulas (\ref{sim1})--(\ref{sim4}) coincides with the homotopy equivalence.

In \cite{GLS} the authors show that if two paths are in the equivalence  relation
defined by (\ref{sim1})--(\ref{sim4}) then they are homotopy equivalent.

For the converse implication we note that the elementary paths $(a,b)$ and $\overline{(b,a)}$ are
1-simplices with the support $b$. Therefore every equivalence
(\ref{sim1})--(\ref{sim4}) is an elementary deformation of paths.
\end{proof}

\section{Mappings and cycles on a Hilbert space}

Let a net of Hilbert spaces
$$
(K,H_a,\gamma_{ba})_{a\leq b\in K}
$$
over  $K$ be given. Here,  $H_a$
is a Hilbert space with a basis $\{e_n^a\}_{n=1}^\infty$, and
$$
\gamma_{ba}:H_a\rightarrow H_b
$$
is an isometric embedding for $a\leq b$,
which transforms the basis $\{e_n^a\}_{n=1}^\infty$ into the basis $\{e_n^b\}_{n=1}^\infty$ and satisfies the equality
$$
\gamma_{ca}=\gamma_{cb}\circ\gamma_{ba},
$$
whenever $a\leq b\leq
c$. If  $a=b$ then $\gamma_{ba}$ is the identity mapping.

For every $a\in K$ we define the set $S_{a} =\left\{p\in S\; \; |\; \;
\partial_{0} p=a\right\}$.

Further, let us consider the Hilbert space of all square summable complex-valued functions on $S_{a}$
$$
l^{2} (S_a)=\bigg\{f:S_a\to \mathbb{C}\; \; |\; \; \sum _{p\in S_a}|f(p)|^{2}  <\infty \bigg\}
$$
with the inner product given by
$$
\left\langle f,g\right\rangle =\sum
_{p\in S_a}f(p)\overline{g(p)}.
$$
The family of functions $\left\{e_{p}
\right\}_{p\in S_a} $ is an orthonormal basis in $l^2(S_a)$.
Here, the equality $e_{p} (p')=\delta _{p,p'} $ holds for $p'\in S_a$, where $\delta _{p,p'}$ stands for the Kronecker symbol.

Let us consider the space $\mathcal{H}=\bigoplus_{a\in K}(H_a\otimes l^2(S_a))$.
For every pair $a,b\in K$ satisfying the condition $a\leq b$ we define the partial isometry $\chi_a^b:\mathcal{H}\rightarrow \mathcal{H}$ as follows:
$$
\chi_a^b(h\otimes e_p)=
 \begin{cases}
 \gamma_{ba}(h)\otimes e_{[\overline{(b,a)}\ast\overline{p}]}, &
 \mbox{if $h\in H_a$ and $\partial_0 p=a$, $\overline{p}\in p$,}
 \\
 0, & \mbox{otherwise.}\\ \end{cases}
$$

We note that the following inclusion holds:
$$
\chi_a^b(H_a\otimes l^2(S_a))\subseteq H_b\otimes
l^2(S_b).
$$
For the operator $\chi_a^{b}$ the conjugate operator
$\chi_a^{b\,*}:\mathcal{H}\rightarrow \mathcal{H}$ is defined by the formula
$$
\chi_a^{b\,*}(h\otimes e_p)=
 \begin{cases}
 h'\otimes e_{[(a,b)\ast\overline{p}]}, & \mbox{if there exists $h'\in H_a$ such that}
 \\
 & \mbox{ $h=\gamma_{ba}(h')$ and $\partial_0 p=b$, $\overline{p}\in p$;}
 \\
 0, & \mbox{otherwise.}\\
 \end{cases}
$$

\begin{lemma}\label{chi} The following assertions hold.

{\rm (1)} $\chi_{a}^c=\chi_{b}^c\chi_{a}^b$ whenever $a\leq b\leq c$.

{\rm (2)} $\chi_{a}^{c\,*}=\chi_{a}^{b\,*}\chi_{b}^{c\,*}$ whenever
$a\leq b\leq c$.

{\rm (3)} $\chi_a^{b\,*}\chi_a^b=I_{H_a}\otimes I_a$, where
$I_{H_a}\otimes I_a:\mathcal{H}\rightarrow H_a\otimes l^2(S_a)$ is
a projection {\rm(}a~surjection\/{\rm)}.

{\rm (4)} $\chi_a^{b}\chi_a^{b\,*}=P_{H_b}\otimes I_b$, where
$P_{H_b}\otimes I_b:\mathcal{H}\hookrightarrow H_b\otimes l^2(S_b)$
is
a projection  {\rm(}an~injection\/{\rm)}.
\end{lemma}
\begin{proof} (1) It follows from the equality $\gamma_{ca}=\gamma_{cb}\circ\gamma_{ba}$ and the relation
$\overline{(c,a)}\ast\overline{p}\sim\overline{(c,b)}\ast\overline{(b,a)}\ast\overline{p}$.

(2) Assume that the condition $\chi_{a}^{c\,*}(h\otimes e_p)=h''\otimes
e_{[(a,c)\ast\overline{p}]}\neq 0$ holds. Then we have
$h=\gamma_{ca}(h'')=\gamma_{cb}\circ\gamma_{ba}(h'')=\gamma_{cb}(h')$.
Consequently, we obtain the equalities
$$
\chi_{a}^{b\,*}\chi_{b}^{c\,*}(h\otimes
e_p)=\chi_{a}^{b\,*}(h'\otimes
e_{[(b,c)\ast\overline{p}]})=h''\otimes
e_{[(a,b)\ast(b,c)\ast\overline{p}]}=h''\otimes
e_{[(a,c)\ast\overline{p}]}.
$$

(3) For every element $h\otimes e_p\in H_a\otimes l^2(S_a)$ one has the equalities
$\chi_a^{b\,*}\chi_a^b(h\otimes
e_p)=\chi_a^{b\,*}(\gamma_{ba}(h)\otimes
e_{[\overline{(b,a)}\ast\overline{p}]})=h\otimes
e_{[(a,b)\ast\overline{(b,a)}\ast\overline{p}]}=h\otimes e_p$.

(4) Take an element $h\otimes e_p\in H_b\otimes l^2(S_b)$. If
 the property $\chi_a^{b}\chi_a^{b\,*}(h\otimes e_p)\neq 0$ is fulfilled then one can easily see that the equality $\chi_a^{b}\chi_a^{b\,*}(h\otimes e_p)=h\otimes e_p$ is valid. Hence, the inclusion
$\chi_a^{b}\chi_a^{b\,*}(H_b\otimes l^2(S_b))\subseteq H_b\otimes
l^2(S_b)$ holds.
\end{proof}

The set of isometries  $\{\chi_a^b\}_{a,b\in K,a\leq b}$ can be enlarged
to the set $\{\chi_{\overline{p}}\}_{\overline{p}\in
\overline{S}}$ as follows. Take an arbitrary sequence of elementary paths
$$
\overline{p}=(a_{2n},a_{2n-1})^{l_{2n-1}}\ast \ldots\ast (a_3,a_2)^{l_2}\ast
(a_2,a_1)^{l_1},
$$
where $l_k=0,1$ and $(a_{k+1},a_{k})^0=(a_{k+1},a_{k}), \
(a_{k+1},a_{k})^1=\overline{(a_{k+1},a_{k})}, \ k=1,...,2n-1$. Then we have
$$
\chi_{\overline{p}}=(\chi_{a_{2n}}^{a_{2n-1}})^{l_{2_n-1}}\ldots(\chi_{a_2}^{a_3})^{l_2}(\chi_{a_2}^{a_1})^{l_1},
$$
where $(\chi_{a_{k+1}}^{a_{k}})^0=\chi_{a_{k+1}}^{a_{k}*}, \
(\chi_{a_{k+1}}^{a_{k}})^1=\chi_{a_{k+1}}^{a_{k}}, \ k=1,...,2n-1$.

We note that the equality $\chi_{\overline{p}}^*=\chi_{\overline{p}^{-1}}$ holds.

The set $\{\chi_{\overline{p}}\}_{\overline{p}\in \overline{S}}$
is closed with respect to the multiplication operation:\break
$\chi_{\overline{p}}\chi_{\overline{q}}=\chi_{\overline{p}\ast\overline{q}}$
provided that the condition  $\partial_1\overline{p}=\partial_0\overline{q}$ holds.

The set
$$
H_{\overline{p}}=\{h\in H_{\partial_1\overline{p}}\,\,|\,\,\chi_{\overline{p}}(h\otimes e_q)\neq 0
\mbox{ if $\partial_0q=\partial_1\overline{p}$}\}
$$
is called \emph{the domain} of the sequence $\overline{p}$.

Obviously, the set $H_{\overline{p}}$ is a Hilbert space. It is worth noting that, in general, we have the condition $H_{\overline{p}}\neq H_{\overline{q}}$ for two distinct sequences $\overline{p},\overline{q}\in
\overline{S}$ , which are equivalent, i.e., $\overline{p}\sim\overline{q}$. For instance, $H_{\overline{p}}=H_a$ if
$\overline{p}=(a,b)\ast\overline{(b,a)}$ with $a\leq b$,
and $H_{\overline{q}}\subseteq H_a$ for
$\overline{q}=\overline{(a,c)}\ast(c,a)$ with $c\leq a$. Here,
the equivalences $\overline{p}\sim
i_a\sim\overline{q}$ hold.

Thus, in general, one has the property $\chi_{\overline{p}}\neq\chi_{\overline{q}}$ for  sequences satisfying the conditions $\overline{p}\sim\overline{q}, \
\overline{p}\neq\overline{q}$.

\begin{lemma}\label{eqv} The following assertions hold{\rm:}

{\rm (1)} the operator $\chi_{i_a}$ is the identity mapping on $H_a\otimes l^2(S_a);$

{\rm (2)} if $\overline{p}\sim \overline{q}$ then
$\chi_{\overline{p}}(h\otimes e_s)=\chi_{\overline{q}}(h\otimes
e_s)$ for every $h\in H_{\overline{p}}\cap H_{\overline{q}}$ and $s$
such that
$\partial_0s=\partial_1\overline{p}=\partial_1\overline{q};$

{\rm (3)} if $\overline{p}\sim \overline{q}$  and
$\gamma_{ba}:H_a\rightarrow H_b$ is an isomorphism for all $a\leq
b\in K$, then the equality $\chi_{\overline{p}}=\chi_{\overline{q}}$ holds.
\end{lemma}
\begin{proof} (1) It is obvious.

(2) It is enough to prove the assertion for equivalences
(\ref{sim1})--(\ref{sim4}).

\noindent To this end, we assume that $a\leq b\leq c$. Then the equivalence $(a,b)\ast
(b,c)\sim (a,c)$ holds. Take an element $h\otimes e_s$ such that
$$
\chi_a^{b*}\chi_b^{c*}(h\otimes e_s)\neq 0\quad\mbox{ and }\quad\chi_a^{c*}(h\otimes
e_s)\neq 0.
$$
It follows from item~2 of Lemma~\ref{chi} that one has the equality
$$
\chi_a^{b*}\chi_b^{c*}(h\otimes e_s)=\chi_a^{c*}(h\otimes e_s).
$$
Item~1 of Lemma~\ref{chi} implies the assertion for the equivalence $\overline{(c,b)}\ast \overline{(b,a)}\sim
\overline{(c,a)}$.
Similarly, for equivalences (\ref{sim3}) and (\ref{sim4})
the assertion follows from items 3 and 4 of Lemma~\ref{chi}.

(3) It follows from (2) that equivalent deformations of the sequence $\overline{p}$ do not change values at points of the space
 $H_{\partial_1\overline{p}}$. Those  deformations only restrict or extend the domain $H_{\overline{p}}$ of the sequence. Furthermore, if
 $\gamma_{ba}:H_a\rightarrow H_b$ is an isomorphism whenever
$a\leq b\in K$, then we have the equalities
$H_{\overline{p}}=H_{\overline{q}}=H_{\partial_1\overline{p}}$.
Consequently, one gets the equality $\chi_{\overline{p}}=\chi_{\overline{q}}$, as required.
\end{proof}

\begin{theorem} The mapping $\pi:\overline{S}\rightarrow
B(\mathcal{H})$ given by
$\pi(\overline{p})=\chi_{\overline{p}}$ is a representation of
$(\overline{S}, \ast)$ in the algebra of bounded operators $B(\mathcal{H})$. If each
embedding $\gamma_{ba}:H_a\rightarrow H_b$
is an isomorphism for $a\leq b\in K$, then the mapping $\pi^*$ defined by $\pi^*([\overline{p}])=\pi(\overline{p})$ is a representation of the groupoid $\overline{S}/_{\sim}$.
\end{theorem}
\begin{proof} Take
$\overline{p},\overline{q}\in\overline{S}$ such that
$\partial_1\overline{p}=\partial_0\overline{q}$. Then we have the equalities
$$
\pi(\overline{p}\ast\overline{q})=\chi_{\overline{p}\ast\overline{q}}=
\chi_{\overline{p}}\chi_{\overline{q}}=\pi(\overline{p})\pi(\overline{q}).
$$
Assume that all  $\gamma_{ba}:H_a\rightarrow H_b$
are isomorphisms for all $a\leq b\in K$. Then, by Lemma~\ref{eqv}, we
have the equality $\chi_{\overline{p}}=\chi_{\overline{q}}$ for $\overline{p}\sim \overline{q}$.  Hence, the mapping $\pi^*$ in the statement of the theorem
is well-defined. Moreover, it is a representation of the groupoid~$\overline{S}/_{\sim}$.
\end{proof}

In the sequel, if a sequence  $\overline{p}$ is a loop, then the mapping $\chi_{\overline{p}}$ is called \emph{a cycle}. We note that the equalities
$$
\chi_{\overline{p}}\chi_{\overline{p}}^*\chi_{\overline{p}}=\chi_{\overline{p}}\,,
\quad
\chi_{\overline{p}}^*\chi_{\overline{p}}\chi_{\overline{p}}^*=\chi_{\overline{p}}^*
$$
hold for every $\overline{p}\in\overline{G}_a$. Therefore the set of cycles
 $\{\chi_{\overline{p}}\}_{\overline{p}\in\overline{G}_a}$
is a regular semigroup. It is clear that the element $\chi_{\overline{p}}^*$ is unique for each
cycle $\chi_{\overline{p}}$. As a consequence, the semigroup of cycles is inverse.

If the equivalence $\overline{p}\sim i_a$ holds for some $a\in K$, then the cycle
$\chi_{\overline{p}}$ is said to be \emph{trivial}.

\begin{theorem}\label{tr} Every trivial cycle
$\chi_{\overline{p}}$ is a projection of the form
$$
\chi_{\overline{p}}=Q_{\overline{p}}\otimes I_a,
$$
where  $\overline{p}\in\overline{G}_a$ and
$Q_{\overline{p}}$ is a projection on the domain
$H_{\overline{p}}$.
\end{theorem}
\begin{proof} By Lemma~\ref{eqv}, since
$\overline{p}\sim i_a$ for an element $a\in K$ the cycle $\chi_{\overline{p}}$ is a projection.
\end{proof}

\begin{corollary}\label{cikl} For every loop $\overline{p}$
the equality
$\chi_{\overline{p}}^*\chi_{\overline{p}}=Q_{\overline{p}}\otimes
I_a$ holds, where $Q_{\overline{p}}$ is a projection on the domain
$H_{\overline{p}^{-1}\ast \overline{p}}=H_{\overline{p}}$.
\end{corollary}
\begin{proof} We note that the equality
$\chi_{\overline{p}}^*\chi_{\overline{p}}=\chi_{\overline{p}^{-1}\ast\overline{p}}$
and the equivalence $\overline{p}^{-1}\ast\overline{p}\sim i_a$ hold for an element $a\in
K$.
\end{proof}

\begin{corollary}\label{pr} If $\overline{p},\overline{q}\in\overline{G}_a$ and
$\overline{p}\sim\overline{q}$, then the operator
$\chi_{\overline{p}}^*\chi_{\overline{q}}=P_{H_{a}}\otimes I_a$
is a projection.
\end{corollary}
\begin{proof} It is sufficient to note that the following equivalences hold:
\begin{equation*}
\overline{p}^{-1}\ast\overline{q}\sim\overline{p}^{-1}\ast\overline{p}\sim
i_a.\qedhere
\end{equation*}
\end{proof}

\begin{corollary} For all trivial cycles $\chi_{\overline{p}}$ and $\chi_{\overline{q}}$ one has the equality
$$
\chi_{\overline{p}}\chi_{\overline{q}}=\chi_{\overline{q}}\chi_{\overline{p}}\,.
$$
\end{corollary}

\begin{theorem}\label{loop} If  $K$ is an upward directed set then every cycle $\chi_{\overline{p}}$
is a projection of the form $\chi_{\overline{p}}=P_{H_{a}}\otimes I_a$, where $\overline{p}\in\overline{G}_a$ and
$P_{H_a}$ is a projection on the domain $H_{\overline{p}}$.
\end{theorem}
\begin{proof} In \cite{GLS}, it is shown that if $K$ is an upward directed set then
for each loop  $\overline{p}$ one has the equivalence $\overline{p}\sim i_a$
for some $a\in K$. This means that every cycle
$\chi_{\overline{p}}$ is trivial. Applying Theorem~\ref{tr}, we obtain the assertion of the theorem.
\end{proof}

A cycle $\chi_{\overline{p}}$ is said to be \emph{finite} if the domain
$H_{\overline{p}}$ is a finite-dimensional linear space.

A cycle $\chi_{\overline{p}}$ is said to be \emph{nilpotent} if
there exists a natural number $m$ such that the equality $\chi_{\overline{p}}^m=0$ holds.

\section{$C^*$-algebras generated by cycles}

In what follows we suppose that the set $K$ is not upward directed.

In general case, for $\overline{p}\in
\overline{G}_a$ every cycle $\chi_{\overline{p}}$ has the form
$\chi_{\overline{p}}=U_{\overline{p}}\otimes T_{\overline{p}}$, where
$U_{\overline{p}}:H_a\rightarrow H_a$ is a partial isometry   and $T_{\overline{p}}:l^2(S_a)\rightarrow l^2(S_a)$ is a unitary operator
corresponding to the loop $\overline{p}$ such that
$T_{\overline{p}}e_q=e_{[\overline{p}\ast\overline{q}]}$, where
$\overline{q}\in q$. By Theorem~\ref{tr}, if a cycle $\chi_{\overline{p}}$ is trivial then we may
write the equality
$\chi_{\overline{p}}=Q_{\overline{p}}\otimes I_a$, where
$Q_{\overline{p}}$ is a projection.

Assume that we are given two equivalent loops $\overline{p}\sim \overline{q}$. Then one has
the equality $T_{\overline{p}}=T_{\overline{q}}$, but, in general, we have
$U_{\overline{p}}\neq U_{\overline{q}}$.  Corollary~\ref{cikl} implies
the equalities
$\chi_{\overline{p}}^*\chi_{\overline{p}}=Q_{\overline{p}}\otimes
I_a$ and
$\chi_{\overline{q}}^*\chi_{\overline{q}}=Q_{\overline{q}}\otimes
I_a$. For loops $\overline{p}\sim \overline{q}$ we define the order relation on cycles: $\chi_{\overline{p}}\leq\chi_{\overline{q}}$
if  $Q_{\overline{p}}\leq Q_{\overline{q}}$. It is easy to verify that one has the relations
\begin{equation}\label{neq}
\chi_{\overline{p}}^*\chi_{\overline{q}}\leq \chi_{\overline{q}}^*\chi_{\overline{q}}\,; \quad
\chi_{\overline{p}}^*\chi_{\overline{q}}\leq
\chi_{\overline{p}}^*\chi_{\overline{p}}\,.
\end{equation}

Indeed, to prove the first relation we rewrite it in the form
 $\chi_{\overline{p}^{-1}\ast\overline{q}}\leq
\chi_{\overline{q}^{-1}\ast\overline{q}} $ and note that
$Q_{\overline{p}^{-1}\ast\overline{q}}\leq
Q_{\overline{q}}=Q_{\overline{q}^{-1}\ast\overline{q}}$. To prove the latter
we make use of Corollary~\ref{pr}. This statement guarantees that
the operator $\chi_{\overline{p}}^*\chi_{\overline{q}}$ is a projection. Hence, we obtain
$\chi_{\overline{p}}^*\chi_{\overline{q}}=\chi_{\overline{q}}^*\chi_{\overline{p}}\leq
\chi_{\overline{p}}^*\chi_{\overline{p}}$, as desired.

For loops $\overline{p}\sim \overline{q}$ we define the addition operation
 $\chi_{\overline{p}}\vee\chi_{\overline{q}}$ of cycles as follows:

1) if the operator $Q_{\overline{p}}+Q_{\overline{q}}$ is also projection, i.e.,
$Q_{\overline{p}}Q_{\overline{q}}=0$, then we set
$\chi_{\overline{p}}\vee\chi_{\overline{q}}=\chi_{\overline{p}}+\chi_{\overline{q}}$;

2) if the condition $Q_{\overline{p}}Q_{\overline{q}}=Q\neq 0$ holds then we put
$\chi_{\overline{p}}\vee\chi_{\overline{q}}=\chi_{\overline{p}}+\chi_{\overline{q}}((Q_{\overline{q}}
- Q)\otimes I_a)$.

\begin{lemma}\label{path} Let $\overline{p}\sim \overline{q}$ be loops with base point
$a$. Then the addition of cycles
$\chi_{\overline{p}}\vee\chi_{\overline{q}}$ can be represented in the form
$\chi_{\overline{p}}\vee\chi_{\overline{q}}=U_{\overline{p},\overline{q}}\otimes
T_{\overline{p}}=U_{\overline{p},\overline{q}}\otimes
T_{\overline{q}}$, where $U_{\overline{p},\overline{q}}:H_a\rightarrow
H_a$ is a partial isometry.
\end{lemma}
\begin{proof} First, we assume that $Q_{\overline{p}}Q_{\overline{q}}=0$. Since
 $T_{\overline{p}}=T_{\overline{q}}$ we get the equalities
$\chi_{\overline{p}}\vee\chi_{\overline{q}}=\chi_{\overline{p}}+\chi_{\overline{q}}=
(U_{\overline{p}}+U_{\overline{q}})\otimes T_{\overline{p}}$, where
$U_{\overline{p}}+U_{\overline{q}}$ is a partial isometry.

Second, we assume that $Q_{\overline{p}}Q_{\overline{q}}=Q\neq 0$.
To prove the lemma it is enough to show that
$(\chi_{\overline{p}}^*\vee\chi_{\overline{q}}^*)(\chi_{\overline{p}}\vee\chi_{\overline{q}})=\widehat{Q}\otimes
I_a$, where $\widehat{Q}$ is a projection. Indeed, using relations
 (\ref{neq}), we have the following:
\begin{equation*}
\begin{split}
(\chi_{\overline{p}}\vee\chi_{\overline{q}})^*(\chi_{\overline{p}}\vee\chi_{\overline{q}})
&=
(\chi_{\overline{p}}^*+((Q_{\overline{q}}- Q)\otimes
I_a)\chi_{\overline{q}}^*)(\chi_{\overline{p}}
+\chi_{\overline{q}}((Q_{\overline{q}}-
Q)\otimes I_a))
\\
&\leq Q_{\overline{p}}\otimes I_a+((Q_{\overline{q}}- Q)\otimes
I_a)(Q_{\overline{p}}\otimes I_a)
\\
&\qquad+(Q_{\overline{p}}\otimes
I_a)((Q_{\overline{q}}- Q)\otimes I_a)
\\
&\qquad+((Q_{\overline{q}}- Q)\otimes I_a)(Q_{\overline{q}}\otimes
I_a)((Q_{\overline{q}}- Q)\otimes I_a)
\\
&=Q_{\overline{p}}\otimes
I_a+(Q_{\overline{q}}- Q)\otimes I_a=\widehat{Q}\otimes I_a.\qedhere
\end{split}
\end{equation*}
\end{proof}

Further, let  $E$ be an infinite subset in an equivalence class
$[\overline{p}]$. We denote by $K(E)$ the family of all finite subsets of the set $E$.  For every $A\in K(E)$ we define the operator
$$
\chi_A=\bigvee\limits_{\overline{q}\in A}\chi_{\overline{q}}.
$$
It follows from Lemma~\ref{path} that $\chi_A$ is a partial isometry  satisfying the property
$$
\chi_A^*\chi_A=Q_A\otimes I_a,
$$
where $Q_A$
is a projection on the space
$$
H_E=\bigcup\limits_{\overline{p}\in
E}H_{\overline{p}}.
$$
As well as it was done for cycles one can define the order relation for all $A,B\in K(E)$ as follows:
$$
\chi_A\leq\chi_B,\quad\mbox{if} Q_A\leq Q_B,
$$
which is equivalent to the inclusion $A\subseteq
B$.

Let $\chi_E$ be the limit with respect to the net $K(E)$ under the inclusion in the strong operator topology. In particular, if
$E=[\overline{p}]$ then we get the operator
$\chi_{[\overline{p}]}=\chi_p$. In the sequel we shall write
$$
\chi_p=\bigvee\limits_{\overline{p}\in p}\chi_{\overline{p}},
$$
where the sum is taken over the whole equivalence class. We shall call this operator \emph{the $p$-cycle}.

In the similar way as it was done for cycles, one can define a finite and a nilpotent
 $p$-cycles. In what follows, we suppose that every $p$-cycle $\chi_p$ is neither finite nor nilpotent.  Although particular cycles $\chi_{\overline{p}}$ may be finite or nilpotent.

\begin{lemma}\label{pcikl} The following assertions hold{\rm:}

{\rm (1)} if  $\overline{p}\sim \overline{q}$ then
$\chi_{[\overline{p}]}=\chi_{[\overline{q}]};$

{\rm (2)} for every $p\in G_a$ the equalities
$\chi_p\chi_p^*\chi_p=\chi_p$ and $\chi_p^*\chi_p\chi_p^*=\chi_p^*$ hold;

{\rm (3)} for every $p,q\in G_a$ the relation
$\chi_p\chi_q\leq\chi_{p\ast q}$ holds.
\end{lemma}
\begin{proof} (1) It follows immediately from the definition of a $p$-cycle.

(2) To prove the first equality we note that the representation
$$
\chi_p^*\chi_p=Q_p\otimes I_a
$$
holds, where $Q_p$ is the projection on the space
$$
H_p=\bigcup\limits_{\overline{p}\in
p}H_{\overline{p}}.
$$
 The proof of the second equation is similar.

(3) It is sufficient to show that the equality
$\chi_p\chi_q=\chi_E$ holds for some $E\subseteq p\ast q$.
Indeed, we have the equalities
$$
\chi_p\chi_q=\bigvee\limits_{\overline{p}\in
p,\overline{q}\in
q}\chi_{\overline{p}}\chi_{\overline{q}}=\bigvee\limits_{\overline{p}\in
p,\overline{q}\in q}\chi_{\overline{p}\ast\overline{q}}.
$$
Then we get
$E=\{\overline{p}\ast\overline{q} \ | \ \overline{p}\in
p,\overline{q}\in q\}\subseteq \{\overline{s} \ | \ \overline{s}\in
p\ast q\}=p\ast q.$ This completes the proof.
\end{proof}

Further, we denote by $\mathfrak{A}_{a,e}$ the subalgebra in $B(\mathcal{H})$ generated by trivial
cycles $\chi_{\overline{p}}$ with $\overline{p}\sim i_a$, which is closed in the strong operator topology.
This algebra acts nontrivially only on the subspace $H_a\otimes
l^2(S_a)$. We notice that this algebra is commutative and contains, in particular, the operators
$\chi_{\overline{p}}^*\chi_{\overline{p}}$\,, $\chi_E^*\chi_E$ for
$E\subseteq [\overline{p}]$, $\chi_p^*\chi_p$ and etc.

Let us consider the family of subspaces
$$
\mathfrak{A}_{a,p}=\mathfrak{A}_{a,e}\chi_p\,, \quad p\in G_a.
$$
The subalgebra~$\mathfrak{A}_{a,e}$ corresponds to the unit $[i_a]$ of the group
$G_a$. We claim that $\mathfrak{A}_{a,p}$ is a Banach space.
Indeed, let us take a Cauchy sequence $\{A_n\chi_p\}_{n=1}^{\infty}$ in
$\mathfrak{A}_{a,p}$. Hence,
$\{A_n\chi_p\chi_p^*\}_{n=1}^{\infty}$ is a Cauchy sequence in~$\mathfrak{A}_{a,e}$ as well. Since $\mathfrak{A}_{a,e}$ is a Banach space the sequence $\{A_n\chi_p\chi_p^*\}$ converges to some element
$B\in\mathfrak{A}_{a,e}$. Then, by Lemma~\ref{pcikl}(2), we have
$A_n\chi_p=A_n\chi_p\chi_p^*\chi_p$ and the sequence $\{A_n\chi_p\}$ converges to the element
$B\chi_p\in \mathfrak{A}_{a,p}$, as claimed.

Let us denote by $\mathfrak{A}_a$ the subalgebra in $B(\mathcal{H})$ generated by elements from the family
$\mathfrak{A}_{a,p}$, $p\in G_a$, which is closed with respect to the uniform norm.

The main result of this paragraph is the proof of the assertion stating that the $C^*$-algebra $\mathfrak{A}_a$ is a
$\pi_1(K)$-graded algebra. For the definition of a $G$-graded
$C^*$-algebra, where $G$
is a group, we refer the reader to \cite[\S 19]{exel}.

\begin{lemma}\label{overp} For every $E\subseteq [\overline{p}]=p$
we have $\chi_E\in\mathfrak{A}_{a,p}$. In particular,
$\chi_{\overline{p}}\in \mathfrak{A}_{a,p}$ for each
$\overline{p}\in p$.
\end{lemma}
\begin{proof} Assume that $E\subseteq p$. Then we have
$\chi_E^*\chi_E\in\mathfrak{A}_{a,e}$ as well as
$\chi_E^*\chi_E\chi_p\in\mathfrak{A}_{a,p}$. Further, one has the equality
$
\chi_E^*\chi_E=Q_E\otimes I_a,
$
where $Q_E$ is a projection on the space
$$
H_E=\bigcup\limits_{\overline{p}\in
E}H_{\overline{p}}\subseteq H_p.
$$
Consequently, we have the equality
$\chi_E^*\chi_E\chi_p=\chi_E$.
\end{proof}

\begin{lemma}\label{pq} For every $p,q\in G_a$ the inclusion $\chi_p\chi_q\in\mathfrak{A}_{a,p\ast
q}$ holds.
\end{lemma}
\begin{proof} In the proof of Lemma~\ref{pcikl} one has already seen that the equality
$$
\chi_p\chi_q=\chi_E
$$
holds for some $E\subseteq p\ast q$.
Hence, by Lemma~\ref{overp}, we have the desired inclusion
\begin{equation*}
\chi_p\chi_q\in\mathfrak{A}_{a,p\ast q}.\qedhere
\end{equation*}
\end{proof}

We recall that  \emph{a conditional expectation} is a positive linear operator $\Phi$ from a $C^*$-algebra $\mathfrak{A}$ to its subalgebra
$\mathfrak{A}_0$ such that $\|\Phi\|=1$ and $\Phi(BAC)=B\Phi(A)C$
for all $B,C\in \mathfrak{A}_0$ and $A\in \mathfrak{A}$.

\begin{lemma}\label{phi} The mapping
$
\Phi:\mathfrak{A}_a\rightarrow\mathfrak{A}_{a,e}
$
given by
$\Phi(B)=B_0$, where
$$
B=
B_0+\sum\limits_{k}B_k\chi_{p_k}\in\bigoplus\limits_{p\in
G_a}\mathfrak{A}_{a,p},\quad B_0,B_k\in\mathfrak{A}_{a,e},\quad p_k\neq[i_a],
$$
is a conditional expectation.
\end{lemma}
\begin{proof} Firstly, let us show that for every $B$ the inequality
$\|\Phi(B)\|\leq\|B\|$ holds. It follows from the definition of the norm that for every $\varepsilon > 0$ there exists an element $h\otimes e_p\in H_a\otimes
l^2(S_a)$ with $\|h\otimes e_p\|=1$ such that we have the estimate
$$
\|B_0(h\otimes
e_p)\|\geq\|B_0\|-\varepsilon.
$$
Then we obtain the following:
\begin{align*}
\|B\|&\geq\|B(h\otimes e_p)\|=\Big\|B_0(h\otimes e_p)+\sum\limits_{k}B_k\chi_{p_k}(h\otimes e_p)\Big\|
\\
&=\Big\|h_0\otimes e_p+\sum\limits_{k_i}h_{k_i}\otimes e_{p_{k_i}\ast p}\|
\geq\|B_0(h\otimes e_p)\Big\|\geq\|B_0\|-\varepsilon.
\end{align*}
The validity of the above-mentioned inequalities follows from the inclusion $\{p_{k_i}\}\subseteq\{p_k\}$ which guarantees the condition
 $p_{k_i}\ast p\neq p$. Since $\varepsilon$ is arbitrary we get
the required estimate
$$
\|\Phi(B)\|=\|B_0\|\leq\|B\|.
$$

Secondly, we take elements
$$
B\in\bigoplus\limits_{p\in G_a}\mathfrak{A}_{a,p},\quad A,C\in
\mathfrak{A}_{a,e}.
$$
Then we have the equalities
\begin{equation*}
\begin{split}
\Phi(ABC)&=\Phi(AB_0C+\sum\limits_{k}AB_k\chi_{p_k}C)
\\
&=\Phi(B_0'+\sum\limits_{k}AB_kC'\chi_{p_k})=B_0'=AB_0C=A\Phi(B)C.\qedhere
\end{split}
\end{equation*}
\end{proof}

\begin{theorem} The $C^*$-algebra $\mathfrak{A}_a$ is  $\pi_1(K)$-graded, that is, the following
representation holds:
    $$
    \mathfrak{A}_a=\overline{\bigoplus\limits_{p\in G_a\cong\pi_1(K)}\mathfrak{A}_{a,p}}\,.
    $$
\end{theorem}
\begin{proof} It is obvious that $\mathfrak{A}_{a,p}\cap\mathfrak{A}_{a,q}={0}$ for $p\neq
q$.

 Let us show that the equality
$\mathfrak{A}_{a,e}\chi_p=\chi_p\mathfrak{A}_{a,e}$ holds. Indeed, we
take an element $P\otimes I_a\in\mathfrak{A}_{a,e}$. Using assertion~(2) in Lemma~\ref{pcikl} together with the commutativity of the algebra
$\mathfrak{A}_{a,e}$, we obtain the following equalities:
$$
(P\otimes I_a)\chi_p=(P\otimes I_a)\chi_p\chi_p^*\chi_p=\chi_p\chi_p^*(P\otimes I_a)\chi_p=\chi_p(Q\otimes I_a),
$$
where
$$
Q\otimes I_a=\chi_p^*(P\otimes
I_a)\chi_p\in\mathfrak{A}_{a,e}.
$$
This yields the desired equality.

By Lemma~\ref{pq}, for every $p,q\in G_a$  we obtain
$$
\mathfrak{A}_{a,p}\mathfrak{A}_{a,q}\subseteq
\mathfrak{A}_{a,e}\chi_p\mathfrak{A}_{a,e}\chi_q=\mathfrak{A}_{a,e}\chi_p\chi_q
\subseteq\mathfrak{A}_{a,p\ast q}.
$$

Further, for $P\otimes
I_a\in\mathfrak{A}_{a,e}$ we have the equalities
$$
((P\otimes I_a)\chi_p)^*=\chi_p^*(P\otimes I_a)=\chi_{p^{-1}}(P\otimes I_a)=(Q\otimes I_a)\chi_{p^{-1}}.
$$
This means that the property
$
(\mathfrak{A}_{a,p})^*=\mathfrak{A}_{a,p^{-1}}
$
is fulfilled.

Finally, Lemma~\ref{phi} implies that there is no an element in the space $\mathfrak{A}_{a,p}$ that can be approximated by linear combinations of elements from the family $\{\mathfrak{A}_{a,q}\}_{q\in
G_a\setminus\{p\}}$.
\end{proof}

\section{Corona for a net of $C^*$-algebras}

The results of the preceding paragraph imply the existence of the family of the
graded $C^*$-algebras $\{\mathfrak{A}_a\}_{a\in K}$
over the set $K$.

 For elements $a\leq b\in K$ we define the mapping $\alpha_{ba}:
\mathfrak{A}_a\rightarrow\mathfrak{A}_b$ as follows. Taking an element
$\chi_{\overline{p}}\in\overline{G}_a$, we set
$$
\alpha_{ba}({\chi_{\overline{p}}})=
\chi_a^b\chi_{\overline{p}}\chi_a^{b*}=\chi_{\overline{(b,a)}\ast\overline{p}\ast(a,b)}.
$$
If  we have the equivalence $\overline{p}\sim i_a$ then we get the equivalence
$\overline{(b,a)}\ast\overline{p}\ast(a,b)\sim i_b$ and the inclusion
$\alpha_{ba}(\mathfrak{A}_{a,e})\subseteq\mathfrak{A}_{b,e}$. Let $p\in G_a$. Then one has the equalities
$$
\alpha_{ba}({\chi_{p}})=
\chi_a^b\bigg(\bigvee\limits_{\overline{p}\in
p}\chi_{\overline{p}}\bigg)\chi_a^{b*}= \bigvee\limits_{\overline{p}\in
p}\chi_{\overline{(b,a)}\ast\overline{p}\ast(a,b)}.
$$
Since the inclusion
$$
\{\overline{(b,a)}\ast\overline{p}\ast(a,b) \ | \ \overline{p}\in
p\}\subseteq
[\overline{(b,a)}\ast\overline{p}\ast(a,b)]=\sigma_{ba}(p)
$$
holds we conclude that
$$
\alpha_{ba}({\chi_{p}})\in\mathfrak{A}_{b,\sigma_{ba}(p)},
$$
where
$\sigma_{ba}:G_a\rightarrow G_b$
is an isomorphism given by formula (\ref{sigma}). Therefore we have the inclusion
$$
\alpha_{ba}(\mathfrak{A}_{a,p})\subseteq
\mathfrak{A}_{b,\sigma_{ba}(p)}.
$$

Thus, the mapping $\alpha_{ba}:
\mathfrak{A}_a\rightarrow\mathfrak{A}_b$ preserves the graduation of the algebras. Moreover, this mapping is an embedding. Really,
using Lemma~\ref{chi}, for all $A,B\in\mathfrak{A}_a$ we obtain the equalities
$$
\alpha_{ba}(AB)=\chi_a^bAB\chi_a^{b*}=\chi_a^bA\chi_a^{b*}\chi_a^bB\chi_a^{b*}=\alpha_{ba}(A)\alpha_{ba}(B).
$$

Lemma~\ref{chi} implies that the property  for the above mappings
$$
\alpha_{ca}=\alpha_{cb}\circ\alpha_{ba}
$$
is fulfilled whenever $a\leq b\leq
c\in K$.

This means that the family of the algebras $\{\mathfrak{A}_a\}_{a\in K}$ constitutes
the net of $C^*$-algebras
$$
\big(K,\mathfrak{A}_a,\alpha_{ba}\big)_{a\leq b\in K}
$$
over the set~$K$, where each mapping $\alpha_{ba}:\mathfrak{A}_a\rightarrow\mathfrak{A}_b$
is an embedding. This net satisfies the isotony property (see
\cite{H}). The algebras of the net will be called \emph{the local algebras}.
We note that if all the mappings $\gamma_{ba}:H_a\rightarrow H_b$
are isomorphisms for $a\leq b\in K$ then the mappings $\alpha_{ba}:\mathfrak{A}_a\rightarrow\mathfrak{A}_b$ are isomorphisms
of algebras.

We represent the partially ordered set $K$ as the union of all its maximal upward directed subsets:
$$
K=\bigcup\limits_{i\in I}K_i.
$$
Such a representation is unique.
Further we consider the net
$$
\big(K_i,\mathfrak{A}_a,\alpha_{ba}\big)_{a\leq b\in K_i}
$$
over the upward directed set $K_i$. Since the mapping $\alpha_{ba}:
\mathfrak{A}_a\hookrightarrow\mathfrak{A}_b$ is an embedding we may assume that
the inclusion $\mathfrak{A}_a\subseteq\mathfrak{A}_b$ holds for $a\leq b$.
We denote by
$$
\mathfrak{A}_i=\overline{\bigcup\limits_{a\in K_i}\mathfrak{A}_a}
$$
the inductive limit of the system of the
$C^*$-algebras $\{\mathfrak{A}_a\}_{a\in K_i}$ over the directed set
 $K_i$, that is, the completion with respect to the unique $C^*$-norm on
$\bigcup_{a\in K_i}\mathfrak{A}_a$. The algebra
$\mathfrak{A}_i$ is called \emph{a quasi-local algebra}.

We call the family of the limit algebras $\{\mathfrak{A}_i\!\}_{i\in I}$
\emph{the corona} for the net of $C^*$\!-algebras
$(K,\mathfrak{A}_a,\alpha_{ba})_{a\leq b\in K}$.

\begin{theorem} In the corona for every $i\in I$ the algebra  $\mathfrak{A}_i $   is a $\pi_1(K)$-graded
   $C^*$-algebra, that is, the following representation holds:
    $$
    \mathfrak{A}_i=\overline{\bigoplus\limits_{p\in \pi_1(K)}\mathfrak{A}_{i,p}}.
    $$
\end{theorem}
\begin{proof} It follows from the fact that the embedding $\alpha_{ba}:
\mathfrak{A}_a\rightarrow\mathfrak{A}_b$ preserves the graduation of the algebras. We have
the representations

\begin{equation*}
\mathfrak{A}_{i,e}=\overline{\bigcup\limits_{a\in
K_i}\mathfrak{A}_{a,e}}\quad\mbox{ and
}\quad\mathfrak{A}_{i,p}=\overline{\bigcup\limits_{a\in
K_i}\mathfrak{A}_{a,p}},\quad p\in\pi_1(K).\qedhere
\end{equation*}
\end{proof}

Assume we are given two nets
$$
\big(K,H^K_a,\gamma_{ba}\big)_{a\leq b\in K}\quad\mbox{ and
}\quad\big(L,H^L_x,\gamma_{yx}\big)_{x\leq y\in L}
$$
over partially ordered sets
 $K$ and $L$, respectively, where $H^K_a$ and $H^L_x$
are Hilbert  spaces and
$\gamma_{ba}:H^K_a\rightarrow H^K_b$ as well as $\gamma_{yx}:H^L_x\rightarrow
H^L_y$
are isometric embeddings for all $a\leq b$ and $x\leq y$.

A pair
$$
(\varphi,\Phi):\big(K,H^K_a,\gamma_{ba}\big)_{a\leq b\in K}\rightarrow
\big(L,H^L_x,\gamma_{yx}\big)_{x\leq y\in L}
$$
is called \emph{a morphism for nets of Hilbert spaces} if the following properties are fulfilled:

1) $\varphi:K\rightarrow L$
is a morphism of partially ordered sets, i.e., the condition $a\leq b$ implies
$\varphi(a)\leq\varphi(b)$;

2) the mapping
$$
\Phi:\bigoplus\limits_{a\in
K}H_a^K\rightarrow\bigoplus\limits_{x\in L}H_x^L
$$
 as well as the mappings
$$
\Phi_a=\Phi
|_{H^K_a}:H^K_a\hookrightarrow H^L_{\varphi(a)}
$$
 for all $a\in K$ are isometric embeddings;

3) the equality
$$
\Phi_b\circ\gamma_{ba}=\gamma_{\varphi(b)\varphi(a)}\circ\Phi_a
$$
holds whenever $a\leq b$.

Similarly, a pair
$$
(\varphi,\Phi):\big(K,\mathfrak{A}_a^K,\alpha_{ba}\big)_{a\leq b\in
K}\rightarrow \big(L,\mathfrak{A}_x^L,\alpha_{yx}\big)_{x\leq y\in L}
$$
is \emph{a morphism for nets of $C^*$-algebras} if
$$
\Phi=\{\Phi_a\}_{a\in K},
$$
 where
$\Phi_a:\mathfrak{A}^K_a\rightarrow\mathfrak{A}^L_{\varphi(a)}$
is a $*$-homomorphism of $C^*$-algebras for every $a\in K$, and the equality
$$
\Phi_b\circ\alpha_{ba}=\alpha_{\varphi(b)\varphi(a)}\circ\Phi_a
$$
holds whenever $a\leq b$. A morphism is said to be \emph{faithful} if  $\Phi_a$
is an embedding for every $a\in K$.

Let $\{\mathfrak{A}_i^K\}_{i\in I}$ and $\{\mathfrak{A}_j^L\}_{j\in
J}$ be the coronas for the nets of  $C^*$-algebras \break
$(K,\mathfrak{A}^K_a,\alpha_{ba})_{a\leq b\in K}$ and
$(L,\mathfrak{A}^L_x,\alpha_{yx})_{x\leq y\in L}$, respectively.
\emph{A morphism of coronas} is a family of mappings
 $\Phi^*=\{\Phi^*_i\}_{i\in I}$ such that for every index $i\in I$
there exists an index $j\in J$ for which
$\Phi_i^*:\mathfrak{A}_i^K\rightarrow\mathfrak{A}_j^L$ is a $*$-homomorphism
of $C^*$-algebras.

A morphism $\varphi:K\rightarrow L$ induces the morphism
$\overline{S}^K\rightarrow\overline{S}^L$, which is denoted by the same letter, as follows:
if  $\overline{p}$
is a sequence of elementary paths of the forms $(a,b)$ and
$\overline{(b,a)}$ on $K$ then we set that $\varphi(\overline{p})$ is a similar
sequence of elementary paths of the forms
$(\varphi(a),\varphi(b))$ and $\overline{(\varphi(b),\varphi(a))}$ on
$L$.

We notice that if  $\overline{p}_1\sim\overline{p}_2$ then
$\varphi(\overline{p}_1)\sim\varphi(\overline{p}_2)$. Therefore,
the morphism $\varphi$ induces the homomorphisms of the groupoids
$$
\varphi^*:\overline{S}^K/_{\sim}\rightarrow\overline{S}^L/_{\sim}
$$
and groups
$$
\varphi^*:G_a^K\rightarrow G_{\varphi(a)}^L
$$
defined by
$\varphi^*([\overline{p}])=[\varphi(\overline{p})]$. Consequently, one gets
the homomorphism of the first homotopy groups
$$
\varphi^*:\pi_1(K)\rightarrow \pi_1(L).
$$

\begin{theorem} Let $\varphi^*:\pi_1(K)\rightarrow \pi_1(L)$
be an injective morphism of the first homotopy groups and $\Phi_a:H^K_a\rightarrow
H^L_{\varphi(a)}$ be an isometric isomorphism for every $a\in
K$. Then the morphism for the nets of Hilbert spaces
$$
(\varphi,\Phi):\big(K,H^K_a,\gamma_{ba}\big)_{a\leq b\in K}\rightarrow
\big(L,H^L_x,\gamma_{yx}\big)_{x\leq y\in L}
$$
induces the faithful morphism for nets of $C^*$-algebras
$$
(\varphi,\Phi^*):\big(K,\mathfrak{A}_a^K,\alpha_{ba}\big)_{a\leq b\in
K}\rightarrow \big(L,\mathfrak{A}_x^L,\alpha_{yx}\big)_{x\leq y\in L}.
$$
\end{theorem}
\begin{proof} Let us consider the direct sums of Hilbert spaces
$$
\mathcal{H}^K=\bigoplus\limits_{a\in K}H_a^K\otimes
l^2(S_a^K)
$$
 and
$$
\mathcal{H}^L=\bigoplus\limits_{x\in L}H_x^L\otimes
l^2(S_x^L).
$$
We define the mapping  $\Phi\otimes
\widehat{\varphi}:\mathcal{H}^K\rightarrow\mathcal{H}^L$ by setting
$$
(\Phi\otimes\widehat{\varphi})(h\otimes e_p)=\Phi(h)\otimes
e_{\varphi^*(p)}
$$
for every $h\otimes e_p\in \mathcal{H}^K$.
It is clear that the mapping
$$
\Phi\otimes\widehat{\varphi}=\bigoplus\limits_{a\in
K}\Phi_a\otimes\widehat{\varphi}
$$
is an isometric embedding
and one has the inclusion
$$
(\Phi_a\otimes\widehat{\varphi})(H_a^K\otimes
l^2(S_a^K))\subseteq H_{\varphi(a)}^L\otimes l^2(S_{\varphi(a)}^L).
$$

We claim that for every $\overline{p}\in\overline{S}$ the equality
$$
\widehat{\varphi}\circ
T_{\overline{p}}=T_{\varphi(\overline{p})}\circ\widehat{\varphi}$$
holds, where $T_{\overline{p}}e_{q}=e_{[\overline{p}\ast\overline{q}]}$.
Indeed, we have the equalities
$$
\widehat{\varphi}
T_{\overline{p}}e_q=\widehat{\varphi}e_{[\overline{p}\ast\overline{q}]}=
e_{[\varphi(\overline{p})\ast\varphi(\overline{q})]}
=T_{\varphi(\overline{p})}e_{\varphi^*(q)}=T_{\varphi(\overline{p})}\widehat{\varphi}e_q,
$$
as claimed.

Let us show the validity of the equality
\begin{equation}\label{hi}(\Phi_b\otimes\widehat{\varphi})\circ\chi_a^b=
\chi_{\varphi(a)}^{\varphi(b)}\circ(\Phi_a\otimes\widehat{\varphi}).\end{equation}
To this end, we write the operator $\chi_a^b$ in the form
$\chi_a^b=\gamma_{ba}\otimes T_{\overline{(b,a)}}$. Then we have the chain of the following equalities:
{\allowdisplaybreaks
\begin{align*}
(\Phi_b\otimes\widehat{\varphi})\chi_a^b&=(\Phi_b\otimes\widehat{\varphi})(\gamma_{ba}\otimes T_{\overline{(b,a)}})
\\
&=
\Phi_b\gamma_{ba}\otimes \widehat{\varphi}T_{\overline{(b,a)}}
=\gamma_{\varphi(b)\varphi(a)}\Phi_a\otimes T_{\overline{(\varphi(b),\varphi(a))}}\widehat{\varphi}
\\
&=
(\gamma_{\varphi(b)\varphi(a)}\otimes
T_{\overline{(\varphi(b),\varphi(a))}})(\Phi_a\otimes
\widehat{\varphi})
=\chi_{\varphi(a)}^{\varphi(b)}(\Phi_a\otimes\widehat{\varphi}).
\end{align*}
}
Since all mappings $\Phi_a$
are isometric isomorphisms we obtain the equality
$\Phi_{\partial_1\overline{p}}(H_{\overline{p}})=H_{\varphi(\overline{p})}$
for $\overline{p}\in\overline{S}$. Hence, it follows from (\ref{hi})
that
$$
(\Phi_a\otimes\widehat{\varphi})\circ\chi_a^{b*}=
\chi_{\varphi(a)}^{\varphi(b)*}\circ(\Phi_b\otimes\widehat{\varphi}).
$$
Therefore, for each $\overline{p}$ one has the equality
$$
(\Phi_{\partial_0\overline{p}}\otimes\widehat{\varphi})\circ\chi_{\overline{p}}=
\chi_{\varphi(\overline{p})}\circ(\Phi_{\partial_1\overline{p}}\otimes\widehat{\varphi}).
$$

We put $\Phi^*_a(\chi_{\overline{p}})=\chi_{\varphi(\overline{p})}$
for every $\overline{p}\in\overline{G}_a$. If  the equivalence
 $\overline{p}\sim i_a$ holds then we have the equivalence $\varphi(\overline{p})\sim
i_{\varphi(a)}$ as well. This means that if the cycle $\chi_{\overline{p}}$
is trivial in the algebra $\mathfrak{A}_a^K$ then
the cycle $\chi_{\varphi(\overline{p})}$ is also trivial in the algebra
$\mathfrak{A}_{\varphi(a)}^L$. Since $\Phi_a$ is an isomorphism
and each trivial cycle has the form
$\chi_{\overline{p}}=Q_{\overline{p}}\otimes I_a$ the mapping
$\chi_{\overline{p}}\mapsto \chi_{\varphi(\overline{p})}$ defined on the set of generators can be extended to the embedding
$\Phi^*_a:\mathfrak{A}_{a,e}^K\rightarrow
\mathfrak{A}_{\varphi(a),e}^L$. Further, we extend the embedding  $\Phi^*_a$ to the whole
algebra $\mathfrak{A}_a^K$ as follows: if
$$
\chi_p=\bigvee\limits_{\overline{p}\in p} \chi_{\overline{p}}
$$
then we set
$$
\Phi^*_a(\chi_{p})=\bigvee\limits_{\overline{p}\in p}
\Phi^*_a(\chi_{\overline{p}})=\bigvee\limits_{\overline{p}\in
p}\chi_{\varphi(\overline{p})}.
$$
Since the condition
$$
\{\varphi(\overline{p})
\ | \ \overline{p}\in
p\}\subseteq[\varphi(\overline{p})]=\varphi^*(p)
$$
holds we get
$$
\Phi^*_a(\chi_{p})\in \mathfrak{A}_{\varphi(a),\varphi^*(p)}^L.
$$
Thus, one has the inclusion
$\Phi^*_a(\mathfrak{A}_{a,p}^K)\subseteq\mathfrak{A}_{\varphi(a),\varphi^*(p)}^L$.
Moreover, because  $\varphi^*$ is an embedding we obtain the embedding
$$
\Phi^*_a:\mathfrak{A}_a^K=\overline{\bigoplus\limits_{p\in\pi_1(K)}\mathfrak{A}_{a,p}^K}
\rightarrow\mathfrak{A}_{\varphi(a)}^L=\overline{\bigoplus\limits_{g\in\pi_1(L)}\mathfrak{A}_{\varphi(a),g}^L}
$$
that preserves the graduation.

It remains to check  the equality
$$
\Phi^*_b\circ\alpha_{ba}=\alpha_{\varphi(b)\varphi(a)}\circ\Phi^*_a
$$
for every $a\leq b\in K$. To do this, we check its validity for the generators. Indeed, for
 $\overline{p}\in \overline{G}_a$ we have
$$
\Phi^*_b\alpha_{ba}(\chi_{\overline{p}})=\Phi^*_b(\chi_a^b\chi_{\overline{p}}\chi_a^{b*})=\chi_{\varphi(a)}^{\varphi(b)}
\chi_{\varphi(\overline{p})}\chi_{\varphi(a)}^{\varphi(b)*}=\alpha_{\varphi(b)\varphi(a)}\Phi^*_a(\chi_{\overline{p}}).
$$
This completes the proof of the theorem.
\end{proof}

\begin{corollary} Let $\{\mathfrak{A}_i^K\}_{i\in I}$ and $\{\mathfrak{A}_j^L\}_{j\in
J}$ be the coronas for  nets
$$
\big(K,\mathfrak{A}^K_a,\alpha_{ba}\big)_{a\leq b\in
K}\quad\mbox{ and }\quad\big(L,\mathfrak{A}^L_x,\alpha_{yx}\big)_{x\leq y\in L},
$$
respectively. Then a morphism for nets of  $C^*$-algebras
$$
(\varphi,\Phi^*):\big(K,\mathfrak{A}_a^K,\alpha_{ba}\big)_{a\leq b\in
K}\rightarrow (L,\mathfrak{A}_x^L,\alpha_{yx})_{x\leq y\in L}
$$
is extended to a morphism of coronas $\Phi^*=\{\Phi^*_i\}_{i\in I}$ so that
for every  $i\in I$ there exists an index $j\in J$ such that
$$
\Phi_i^*(\mathfrak{A}_i^K)\subseteq\mathfrak{A}_j^L.
$$
\end{corollary}
\begin{proof} Let us consider the inductive limit
$$
\mathfrak{A}_i^K=\overline{\bigcup\limits_{a\in
K_i}\mathfrak{A}_a^K}.
$$
We put
$$
\Phi_i^*(\mathfrak{A}_i^K)=\overline{\bigcup\limits_{a\in
K_i}\Phi_a^*(\mathfrak{A}_a^K)}\subseteq\overline{\bigcup\limits_{a\in
K_i}\mathfrak{A}_{\varphi(a)}^L}.
$$
Since the inclusion $\varphi(K_i)\subseteq
L_j$ holds for some index $j\in J$ we have
$\Phi_i^*(\mathfrak{A}_i^K)\subseteq \mathfrak{A}_j^L$, as required.
\end{proof}

The following example demonstrates that if  $\varphi^*$ is not an embedding
then a morphism for nets of $C^*$-algebras is not faithful, in general.

\begin{Ex}{Example} Let $Y$ be the open unit disk in the complex plane with
 the center at the coordinate origin and $X=Y\backslash\{(0,0)\}$.
Let $K$ and $L$
be the families of all open simply connected subsets of the sets $X$ and
$Y$, respectively. The families $K$ and $L$ are partially ordered sets under the inclusion relation. Moreover, the set $L$ is upward directed. It is easy to see that the equalities $\pi_1(K)=\mathbb{Z}$ and
$\pi_1(L)=\{0\}$ hold. The inclusion $X\subseteq Y$
yields the embedding $\varphi: K\rightarrow L$. Therefore, we have the
homomorphism of the first homotopy groups
$\varphi^*:\pi_1(K)\rightarrow\pi_1(L)$ such that $\varphi^*(n)=0$
for each $n\in\mathbb{Z}$.

Further, let $H$ be a Hilbert space. We consider the bundles of Hilbert spaces
 $(K,H_a,\gamma_{ba})_{a\leq b\in
K}$ and $(L,H_x,\gamma_{yx})_{x\leq y\in L}$ over $K\!$ and $L\!$, respectively (see~\cite{RuzziVasselli}), where $H_a=H_x=H$ and $\gamma_{ba}$,
$\gamma_{yx}$
are the identity mappings. Then one associates to them the bundles  of $C^*$-algebras
$(K,\mathfrak{A}_a^K,\alpha_{ba})_{a\leq b\in K}$ and
$(L,\mathfrak{A}_x^L,\alpha_{yx})_{x\leq y\in L}$ over $K$ and $L$, respectively,
where $\alpha_{ba}$ as well as $\alpha_{yx}$
are isomorphisms. The $C^*$-algebra~$\mathfrak{A}_a^K$ is
$\mathbb{Z}$-graded. It is generated by the operators
 $\chi_n=I\otimes T_{0}^n$, $n\in \pi_1(K)$. Here
$I\otimes T_{0}$ is the unitary two-sided shift operator on the space
$H\otimes l^2(S^K_a)$, which corresponds to
$n=1\in\pi_1(K)$. Therefore, we have the isomorphism $\mathfrak{A}_a^K\cong C(S^1)$, where
$C(S^1)$
is the Banach algebra of all continuous complex-valued functions on the unit circle in the complex plane. The $C^*$-algebra~$\mathfrak{A}_x^L$ is generated by the operator
$\chi_{\varphi^*(n)}=I\otimes I_a$. Hence, one has the isomorphism $\mathfrak{A}_x^L\cong
\mathbb{C}$. The mapping
$\Phi_a^*:\mathfrak{A}_a^K\rightarrow\mathfrak{A}_{\varphi(a)}^L$ defined by
the correspondence  $\chi_n\mapsto \chi_{\varphi^*(n)}$ yields the following commutative diagram:
$$
\begin{CD}
\mathfrak{A}_a^K @>\Phi_a^*>> \mathfrak{A}_{\varphi(a)}^L \\
@V\cong VV @VV\cong V \\
C(S^1) @>m>>\mathbb{C}
\end{CD},
$$
where $m:C(S^1)\rightarrow\mathbb{C}$ is the multiplicative functional given by
$$
{m(f)=f(1)}.
$$
Obviously, the mapping $m$ is not an embedding.
\end{Ex}

\end{document}